\documentclass[fleqn]{zzz}
\usepackage{mathrsfs}
\usepackage{color}
\usepackage{hyperref}

\hypersetup{
   colorlinks = true,
   urlcolor = blue,
   linkcolor = red
}

\newcommand{\N}{{\mathbb N}}
\newcommand{\R}{{\mathbb R}}
\newcommand{\C}{{\mathbb C}}

\newtheorem{thm}[lemma]{Theorem}
\newtheorem{cor}[lemma]{Corollary}
\newtheorem{prop}[lemma]{Proposition}

\makeatletter
\def\eqnarray{\stepcounter{equation}\let\@currentlabel=\theequation
\global\@eqnswtrue
\tabskip\@centering\let\\=\@eqncr
$$\halign to \displaywidth\bgroup\hfil\global\@eqcnt\z@
  $\displaystyle\tabskip\z@{##}$&\global\@eqcnt\@ne
  \hfil$\displaystyle{{}##{}}$\hfil
  &\global\@eqcnt\tw@ $\displaystyle{##}$\hfil
  \tabskip\@centering&\llap{##}\tabskip\z@\cr}

\def\endeqnarray{\@@eqncr\egroup
      \global\advance\c@equation\m@ne$$\global\@ignoretrue}

\def\@yeqncr{\@ifnextchar [{\@xeqncr}{\@xeqncr[5pt]}}
\makeatother

\begin{document}

\renewcommand{\PaperNumber}{***}

\FirstPageHeading

\ShortArticleName{Sums involving the number of distinct prime factors function}
\ArticleName{On sums involving the number of distinct prime factors function}

\Author{Tanay V.~Wakhare$^{\ast\dagger}$}

\AuthorNameForHeading{T.~Wakhare}

\Address{$^\ast$~University of Maryland, College Park MD 20742}
\Address{$^\dagger$~Center for Nanoscale Science and Technology, National Institute of Standards and Technology, Gaithersburg MD 20899}
\EmailD{twakhare@gmail.com}

\ArticleDates{Received XX January 2017 in final form ????; Published online ????}

\Abstract{The main object of this paper is to find closed form expressions for finite and infinite sums that are weighted by $\omega(n)$, where $\omega(n)$ is the number of distinct prime factors of $n$. We then derive general convergence criteria for these series. The approach of this paper is to use the theory of symmetric functions to derive identities for the elementary symmetric functions, then apply these identities to arbitrary primes and values of multiplicative functions evaluated at primes. This allows us to reinterpret sums over arbitrary complex numbers as divisor sums and sums over the natural numbers.
}


\Classification{11C08;11M06;11N99;11Z05}

\section{Introduction}
\label{Introduction}
The number of distinct prime factors of a natural number $n$, denoted by $\omega(n)$, has been the subject of extensive study, from Hardy and Ramanujan \cite{HardyRamanujan} to Erd{\H{o}}s and Kac \cite{ErdosKac}. Letting $n=\prod_i p_i^{\alpha_i}$, we have $\omega(n)=i$. For example, $\omega(5)=1$, $\omega(10)=2$, and $\omega(100)=\omega(2^25^2)=2$. Throughout, we assume that $\omega(1)=0$. The function $\omega(n)$ is also additive, so $\omega(nm)=\omega(n)+\omega(m)$ for any coprime $n$ and $m$. 

In this work, we address two major types of identities: divisor sums and Dirichlet series. Many of the closed form expressions available here do not seem to be present anywhere. This is because the easiest functions to work with are \textit{multiplicative}, while $\omega(n)$ is \textit{additive}. Identities mixing both types of functions are not easy to derive using standard methods. In this paper, we use the theory of symmetric functions to find an alternative proof for many series expansions. The methods readily generalize, and can be applied to other additive functions. Our two main results are that
\begin{equation}
\sum_{d|n}\mu(d)\omega(d) f(d) = \left(\sum_{d|n}\mu(d) f(d)\right)\left( \sum_{p|n}\frac{f(p)}{f(p)-1}\right)
\end{equation}
and
\begin{equation}\label{inf1}
\sum_{n\in\N}\frac{\omega(n)f(n)}{n^s} = \left(\sum_{n\in\N}\frac{f(n)}{n^s}\right)\left(\sum_{p}\frac{a_p}{1+a_p}\right),
\end{equation}
with suitable restrictions on $f$ and $s$, and $a_p:=\sum_{m=1}^{\infty}\frac{f({p}^m)}{{p}^{ms}}$.

For convenience, we note here that a multiplicative function satisfies $f(nm)=f(n)f(m)$ for $(n,m)=1$, where $(n,m)$ gives the greatest common divisor of $n$ and $m$. A completely multiplicative function satisfies $f(nm)=f(n)f(m)$ for any $n$ and $m$. Both satisfy $f(1)=1$. An additive function satisfies $f(nm)=f(n)+f(m)$ for $(n,m)=1$, while a completely additive function satisfies $f(nm)=f(n)+f(m)$ for any $n$ and $m$.

Now, we address some notation. We let $\sum_{p}a_p$ and $\prod_{p}a_p$ denote sums and products over all primes $p$, beginning with $p=2$. We let $\sum_{p|n}a_p$ and $\prod_{p|n}a_p$ denote sums and products over the distinct primes that divide a positive integer $n$. Finally, we let $\sum_{d|n}a_d$ and $\prod_{d|n}a_d$ denote sums and products over the positive divisors of a positive integer $n$, including $1$ and $n$. For example if $n=12$, $\prod_{d|12}a_d=a_1a_2a_3a_4a_6a_{12}$ and $\prod_{p|12}a_p= a_2a_3$. For the duration of this paper, the symbol $\N:=\left\lbrace 1,2,3\ldots \right\rbrace$ will denote the set of positive integers and will be referred to as the natural numbers. We obey the convention that for $s \in \C$, $s= \sigma + i t$ with $\sigma, t \in \R$. We also define $\C_{\alpha}:=\{z\in\C: z\ne \alpha\}$. 

We also utilize the theory of Euler products. An Euler product, described in \cite[(27.4.1-2)]{NIST:DLMF}, is the product form of a Dirichlet series. For any multiplicative function $f$ we have \cite[(11.8)]{Apostol} $\sum_{n\in\N}\frac{f(n)}{n^s}=\prod_{p}\left(1+a_p\right)$, where
\begin{equation}\label{infap}
a_p:=\sum_{m=1}^{\infty}\frac{f({p}^m)}{{p}^{ms}}.
\end{equation}
The left-hand side is an Dirichlet series, and the right-hand is an Euler product. Furthermore, letting $s=\sigma+it$, the abscissa of absolute convergence is the unique real number such that $\sum_{n\in\N}\frac{f(n)}{n^s}$ absolutely converges if and only if $\sigma>\sigma_a$. 

Some multiplicative functions which will be used in this paper include $\phi(n)$, which denotes the Euler totient function. This counts the number of natural numbers less than or equal to $n$ which are coprime to $n$\cite[(27.2.7)]{NIST:DLMF}. A multiplicative expression for it \cite[(27.3.3)]{NIST:DLMF} is
\begin{equation}\label{def1}
\phi(n)=n\prod_{p|n}\left(1-\frac{1}{p}\right).
\end{equation}
This is generalized by Jordan's totient function, which is defined in \cite[(27.3.4)]{NIST:DLMF} as
\begin{equation}\label{def2}
J_k(n):=n^k\prod_{p|n}\left(1-\frac{1}{p^k}\right).
\end{equation}

Using $\omega(n)$, we can also define the M{\"o}bius function, which gives the parity of the number of prime factors in a squarefree number. It is defined in \cite[(27.2.12)]{NIST:DLMF} as 
\begin{equation}\label{def4}
\mu(n) :=\begin{cases} 1, & \text{if $n$=1,} \\ 0, & \text{if $n$ is non-squarefree,} \\  (-1)^{\omega(n)}, & \text{if $n$ is squarefree.} \\\end{cases}
\end{equation}

The paper is organized as follows.
In Section \ref{polynomials}, we derive identities for the elemtary symmetric functions.
In Section \ref{divisorsums}, divisor sums of $\omega(n)$ weighted by other functions are explored.
In Section \ref{infinitesums}, Dirichlet series of $\omega(n)$ weighted by other functions are explored.
Finally, in Section \ref{extensions} we present an extension to higher orders, and lay out other possible extensions of this work.


\section{Factorization identities}
\label{polynomials}
The main results of this paper are based on the following proposition.
\begin{prop} Let $n\in\N$ and $x_1,\ldots,x_n \in \C_1$. Then
\begin{equation}\label{main}
\left(\prod_{i = 1}^{n}(1-x_i)\right)\left( \sum_{i = 1}^{n}\frac{x_i}{x_i-1}\right)= \sum_{k = 1}^{n}(-1)^k k e _{k}, 
\end{equation}
where 
\begin{equation}
e_k:=\sum_{1\leqslant i_1< i_2< \cdots < i_k \leqslant n }^{}x_{i_1}x_{i_2}\cdots x_{i_k}.
\end{equation}
\end{prop}
\begin{proof}
Following Macdonald \cite[(Chapter 1)]{Macdonald}, we have the relation
\begin{equation}\label{VietaDef}
E(t)=\prod_{i = 1}^{n}(1- t x_i)=\sum_{k = 0}^{n}(-1)^{k}t^k e_k,
\end{equation}
with $x_i\in\C$. Taking a logarithmic derivative of $E(t)$ yields
\begin{equation}
E'(t)=E(t)\sum_{k=1}^n \frac{x_i}{t x_i-1}.
\end{equation}
Taking $t=1$ completes the proof.
\end{proof}

\begin{cor} Let $n\in\N$ and $x_1,\ldots,x_n \in \C_{-1}$. Then
\begin{equation}\label{mainpos}
\left(\prod_{i = 1}^{n}(1+x_i)\right)\left(\sum_{i = 1}^{n}\frac{x_i}{x_i+1}\right)=\sum_{k = 1}^{n}k e _{k}.
\end{equation}
\end{cor}
\begin{proof}
We repeat the previous argument with but map $t$ to $-t$ before taking the logarithmic derivative.
\end{proof}

\section{Divisor sums}\label{divisorsums}
The following theorems are obtained by reinterpreting \eqref{main}. We consider divisor sums of a multiplicative function $f$ weighted by $\omega(n)$ and $\mu(n)$ or $|\mu(n)|$. Throughout, $f(p)$ will refer to the value of $f$ evaluated at any prime $p$. We define $\sum_{p|n} g(p)$ as $0$ if $n=1$, where $g$ is any function, not necessarily multiplicative, because $1$ has no distinct prime factors. Under this convention, the theorems in this section also hold for $n=1$.

\begin{thm} Let $f(n)$ be a multiplicative function, $n \in \N$, and $f(p)\neq 1$ for any prime $p$ that divides $n$. Then
\begin{equation}\label{finite1}
\sum_{d|n}\mu(d)\omega(d) f(d) = \left(\prod_{p|n}(1-f(p))\right) \left(\sum_{p|n}\frac{f(p)}{f(p)-1}\right).
\end{equation}
\end{thm}
\begin{proof}
First take any squarefree natural number $n$, so that $n = \prod_{i=1}^{\omega(n)}{p_i}$ by the fundamental theorem of arithmetic. We then let $x_i=f(p_i)$ in \eqref{main}, such that each $x_i$ is an arithmetic function evaluated at each distinct prime that divides $n$. We can evaluate $e_k$, yielding
$$e_k=\sum_{1\leqslant i_1< i_2 \cdots < i_k \leqslant n }f(p_{i_1})f(p_{i_2})\cdots f(p_{i_k}).$$
Since $f$ is multiplicative and each $p_i$ is coprime to the others by definition, we have 
$$e_k=\sum_{1\leqslant i_1< i_2 \cdots < i_k \leqslant n }f(p_{i_1}p_{i_2}\cdots p_{i_k}).$$
Now we can regard $e_k$ as the sum of $f(n)$ evaluated at the divisors of $n$ with $k$ prime factors. This is because each term in $e_k$ trivially has $k$ prime factors, and every possible product of $k$ primes that divide $n$ is included in $e_k$. This is equivalent to partitioning the divisors of $n$ based on their number of distinct prime factors. Then \eqref{main} transforms into 
$$\left(\prod_{i = 1}^{\omega(n)}(1-f(p_i))\right)\left( \sum_{i = 1}^{\omega(n)}\frac{f(p_i)}{f(p_i)-1}\right)=\sum_{d|n}\omega(d)(-1)^{\omega(d)}f(d).$$
Each divisor $d$ is squarefree since $n$ is squarefree, so we can replace $(-1)^{\omega(d)}$ by $\mu(d)$, where $\mu(d)$ is the M{\"o}bius function which is defined by $\eqref{def4}$. Rewriting the product and sum over $p_i$, $1\leq i \leq \omega(n)$, as a product and sum over $p$ gives \eqref{finite1} for squarefree numbers. However, we can immediately see that if $n$ is non-squarefree, $\mu(d)$ eliminates any non-squarefree divisors on the left-hand side. Meanwhile, the right-hand side is evaluated over the distinct primes that divide $n$ so changing the multiplicities of these primes will not affect the sum in any way. Therefore  \eqref{finite1} is valid for all $n\in\N$, which completes the proof.
\end{proof}

\begin{cor} Let $f(n)$ be a multiplicative function, $n \in \N$, and $f(p)\neq 1$ for any prime $p$ that divides $n$. Then
\begin{equation}\label{finitecor1}
\sum_{d|n}\mu(d)\omega(d) f(d) = \left(\sum_{d|n}\mu(d) f(d)\right)\left( \sum_{p|n}\frac{f(p)}{f(p)-1}\right).
\end{equation}
\end{cor}
\begin{proof}
A general theorem for any multiplicative function $f(n)$, found in \cite[(2.18)]{Apostol}, is that
\begin{equation}\label{multApostol}
\sum_{d|n}\mu(d) f(d)= \prod_{p|n}\left(1-f(p)\right).
\end{equation}
Substituting this relation into \eqref{finite1} completes the proof.
\end{proof}

\begin{thm} Let $f(n)$ be a multiplicative function, $n \in \N$, and $f(p)\neq -1$ for any prime $p$ that divides $n$. Then
\begin{equation}\label{finite2}
\sum_{d|n}|\mu(d)|\omega(d) f(d) = \left(\prod_{p|n}(1+f(p))\right)\left( \sum_{p|n}\frac{f(p)}{1+f(p)}\right).
\end{equation}
\end{thm}
\begin{proof}
We follow the reasoning of the proof of \eqref{finite1} but substitute $x_i=f(p_i)$ into \eqref{mainpos} instead of \eqref{main}. Hence, we only have to  take the divisor sum over squarefree divisors without multiplying by $(-1)^{\omega(d)}$. We do this by multiplying the divisor sums by $|\mu(d)|$, the characteristic function of the squarefree numbers. This completes the proof.
\end{proof}

\begin{cor} Let $f(n)$ be a multiplicative function with $f(1)=1$, $f(p)\neq -1$, and $n\in\N$. Then
\begin{equation}
\sum_{d|n}|\mu(d)|\omega(d) f(d) =\left( \sum_{d|n}|\mu(d)|f(d)\right)\left( \sum_{p|n}\frac{f(p)}{1+f(p)}\right).
\end{equation}
\end{cor}
\begin{proof}
We substitute $f(n)=\mu(n)g(n)$ into \eqref{multApostol}, where $g(n)$ is multiplicative, ensuring that $f(n)$ is also multiplicative. Noting that $\mu^2(n)=|\mu(n)|$ since the M{\"o}bius function only takes values of $\pm1$ and $0$ yields
$$\sum_{d|n}|\mu(d)| g(d)= \prod_{p|n}\left(1+g(p)\right).$$
Mapping $g$ to $f$ to maintain consistent notation and substituting into \eqref{finite2} completes the proof.
\end{proof}

Specializing $f(n)$ yields a variety of new formulae involving convolutions with $\omega(n)$. Below we use a variety of functions $f(n)$ along with \eqref{finite1}, \eqref{finitecor1}, and \eqref{finite2} to find new expressions for divisor sums involving $\omega(n)$.

\begin{thm} Let $n\in\N$. Then
\begin{equation}\label{finite4}
\sum_{d|n}|\mu(d)|\omega(d) = \omega(n) 2^{\omega(n)-1}.
\end{equation}
\end{thm}
\begin{proof}
Substituting $f(n)=1$ into \eqref{finite2} gives $\sum_{d|n}|\mu(d)|\omega(d) = \left(\prod_{p|n}2\right)\left( \sum_{p|n}\frac{1}{2}\right).$
Since the product and sum on the right-hand side are over the distinct primes that divide $n$, each is evaluated $\omega(n)$ times. This simplifies to $\sum_{d|n}|\mu(d)|\omega(d) =2^{\omega(n)} \frac{\omega(n)}{2}.$
\end{proof}

\begin{thm} Let $n\in\N$. Then
\begin{equation}\label{finite6}
\sum_{d|n}\mu(d)\omega(d){\left(\frac{n}{d}\right)}^k = {J_k(n)}\sum_{p|n}\frac{1}{1-p^k}.
\end{equation}
\end{thm}
\begin{proof}
Substituting $f(n)=\frac{1}{n^k}$ into \eqref{finite1} gives
\begin{equation}\label{jordanOne}
\sum_{d|n}\frac{\mu(d)\omega(d)}{d^k} = \left(\prod_{p|n}\left(1-\frac{1}{p^k}\right) \right)\left(\sum_{p|n}\frac{1}{1-p^k}\right).
\end{equation}
Using \eqref{def2} to see that $\prod_{p|n}(1-\frac{1}{p^k}) = \frac{J_k(n)}{n^k}$ and substituting this relation into \eqref{jordanOne} completes the proof.
\end{proof}

\begin{thm} Let $n\in\N$. Then
\begin{equation}\label{finite9}
\sum_{d|n}|\mu(d)|\omega(d){\left(\frac{n}{d}\right)}^k = \frac{J_{2k}(n)}{J_k(n)}\sum_{p|n}\frac{1}{1+p^k}.
\end{equation}
\end{thm}
\begin{proof}
Substituting $f(n)=\frac{1}{n^k}$ into \eqref{finite2} gives
$$\sum_{d|n}\frac{|\mu(d)|\omega(d)}{d^k} =\left(\prod_{p|n}\left(1+\frac{1}{p^k}\right) \right)\left(\sum_{p|n}\frac{1}{1+p^k}\right).$$
Now we complete the proof by using \eqref{def2} to see that 
$$\prod_{p|n}\left(1+\frac{1}{p^k}\right)=\frac{n^{2k}}{n^{2k}}\frac{\prod_{p|n}(1-\frac{1}{p^{2k}})}{\prod_{p|n}(1-\frac{1}{p^k})}=\frac{1}{n^k}\frac{J_{2k}(n)}{J_k(n)}.$$
\end{proof}

\begin{cor} Let $n\in\N$. Then
\begin{equation}
\sum_{d|n}|\mu(d)|\omega(d){\left(\frac{n}{d}\right)} = \psi(n)\sum_{p|n}\frac{1}{1+p}.
\end{equation}
\end{cor}
\begin{proof}
Substituting $k=1$ into \eqref{finite9} and noting that $\psi(n)=\frac{J_{2}(n)}{J_1(n)}$ completes the proof.
\end{proof}

\begin{thm} Let $n\in\N$ with $n$ squarefree. Then
\begin{equation}
\sum_{d|n}\omega(d)d^k =  \frac{J_{2k}(n)}{J_k(n)} \sum_{p|n}\frac{p^k}{1+p^k}.
\end{equation}
\end{thm}
\begin{proof}
Substituting $f(n)=n^k$ into \eqref{finite2} gives
\begin{equation}\label{jordanTwo}
\sum_{d|n}|\mu(d)|\omega(d)d^k = \prod_{p|n}(1+{p^k}) \sum_{p|n}\frac{p^k}{1+p^k}.
\end{equation}
Now if we let $n$ be squarefree, then $n=\prod_{p|n} p$. We also know that \eqref{def2} transforms into 
$$J_k(n)=n^k\prod_{p|n}\left(1-\frac{1}{p^k}\right)=n^k \frac{\prod_{p|n}(p^k-1)}{\prod_{p|n}p^k}.$$ 
However we now have $n^k=\prod_{p|n}p^k$, so $J_k(n)=\prod_{p|n}(p^k-1)$. Then 
\begin{equation}\label{jordanThree}
\frac{J_{2k}(n)}{J_k(n)}=\frac{\prod_{p|n}(p^{2k}-1)}{\prod_{p|n}(p^k-1)}=\prod_{p|n}\frac{(p^{2k}-1)}{(p^k-1)}=\prod_{p|n}(p^k+1).
\end{equation}
Letting $n$ be a squarefree natural number in \eqref{finite2}, we can also eliminate $|\mu(d)|$ from the sum on the left-hand side since every divisor of $n$ will already be squarefree. Substituting \eqref{jordanThree} into \eqref{jordanTwo} completes the proof.
\end{proof}

\section{Infinite sums}\label{infinitesums}
The propositions \eqref{main} and \eqref{mainpos} hold for a finite number of elements $x_i$. However, we can take the limit $i\rightarrow\infty$ to extend this sum. This enables us to find closed form product expressions for Dirichlet series, which are described in \cite[(27.4.4)]{NIST:DLMF}, of the form $\sum_{n\in\N}\frac{\omega(n)f(n)}{n^s}$ for many commonly encountered multiplicative functions $f(n)$.

For later convenience we introduce the prime zeta function, described in Fr{\"o}berg (1968) \cite[(0.1)]{PrimeZeta}, denoted by $P(s)$. We define it by
$$P(s) := \sum_{p}\frac{1}{p^s},$$
and note that it converges for $\Re(s)>1$. It is an analog of the Riemann zeta function, described in \cite[(25.2.1)]{NIST:DLMF}, with the sum taken over prime numbers instead of all natural numbers. For notational convenience we also define the shifted prime zeta function $P(s,a)$ as 
$$P(s,a) := \sum_{p}\frac{1}{p^s+a},$$
such that $P(s,0)=P(s)$.

\begin{lemma}\label{primezetacor}
Let $a\in\C$ and $|a|< 2$. Then P(s,a) converges absolutely if and only if $s\in\C$, $\Re(s)>1$. 
\end{lemma}
\begin{proof}
The result follows from using a direct comparison test with $P(1)$ to prove the divergence of $P(1,a)$, then taking an absolute value to bound it above and prove absolute convergence for $\Re(s)>1$.
\end{proof}

\begin{lemma}\label{primezetacor2}
Let $a,s,k\in\C$ with $|a|<2$. Then $\sum_{p} \frac{p^k}{p^s+a}$ converges absolutely if and only if $\Re(s)>\max\left(1,1+\Re(k)\right)$.
\end{lemma}
\begin{proof}
The result follows from taking an absolute value and multiplying the top and bottom by $p^{-k}$, then applying Lemma \ref{primezetacor}, whether $\Re(k)\leq0$ or not.
\end{proof}

\begin{thm}\label{infmain} Let $f(n)$ be a multiplicative function, $s\in\C$,  $a_p:=\sum_{m=1}^{\infty}\frac{f({p}^m)}{{p}^{ms}}$, and $a_p \neq -1$ for any prime $p$. If $\sum_{n\in\N}\frac{f(n)}{n^s}$ and $\sum_{p}\frac{a_p}{1+a_p}$ both converge absolutely for $\sigma>\sigma_a$, then
\begin{equation}\label{inf1}
\sum_{n\in\N}\frac{\omega(n)f(n)}{n^s} = \left(\sum_{n\in\N}\frac{f(n)}{n^s}\right)\left(\sum_{p}\frac{a_p}{1+a_p}\right),
\end{equation}
which converges absolutely for $\sigma>\sigma_a$.
\end{thm}
\begin{proof}
We let $p_i$ denote the $i$\textsuperscript{th} prime number. We must choose a suitable $x_i$ to substitute into \eqref{mainpos}, so we let
$x_i = a_{p_i}$, $1\leq i<\infty$. We still retain the condition $x_i=a_{p_i}\neq -1$ to avoid dividing by $0$. Substituting this $x_i$ into \eqref{mainpos} gives
$$\sum_{k=1}^{\infty}ke_k = \prod_{p}\left( 1 + a_p \right) \sum_{p}\frac{a_p}{1+a_p.}$$
The product and sum on the right now go through every prime $p$ since each $x_i$ is in a one-to-one correspondence with a sum over the $i$th prime. We also have that the product over primes is the Euler product for the Dirichlet series $\sum_{n\in\N}\frac{f(n)}{n^s}$. We now prove that $e_k$ sums over every natural number $n$ that has $k$ distinct prime factors.

\begin{lemma}
Let $1\leq i<\infty$, $i\in\N$, $p_i$ denote the $i$\textsuperscript{th} prime, $f$ denote any multiplicative function, and $x_i=\sum_{m=1}^{\infty}\frac{f({p_i}^m)}{{p_i}^{ms}}$.
Furthermore, let $k\in\N$ and
\begin{equation}\label{defS}
 S_k:=\{n\in\N: \omega(n)=k\},
\end{equation}
so that $S_k$ is the set of natural numbers with $k$ distinct prime factors. If $\sum_{n\in\N}\frac{f(n)}{n^s}$  has an abscissa of absolute convergence $\sigma_a$, then
\begin{equation}\label{inf2}
e_k:=\sum_{1\leqslant i_1< i_2< \cdots < i_k \leqslant n }^{}x_{i_1}x_{i_2}\cdots x_{i_k} = \sum_{n \in S_k}\frac{f(n)}{n^s},
\end{equation}
which converges absolutely for $\sigma>\sigma_a$.
\end{lemma}

\begin{proof}
Let $1\leq i < \infty$, $i\in\N$, and $1\leq\epsilon_i<\infty$, $\epsilon_i\in\N$. We can directly evaluate $e_k$ as
$$e_k=\sum_{1\leqslant i_1< i_2 \cdots < i_k \leqslant n }\frac{f(p_{i_1}^{\epsilon_{i_1}})f(p_{i_2}^{\epsilon_{i_2}})\cdots f(p_{i_k}^{\epsilon_{i_k}})}{\left(p_{i_1}^{\epsilon_{i_1}}p_{i_2}^{\epsilon_{i_2}}\cdots p_{i_k}^{\epsilon_{i_k}}\right)^s}.$$
Here $\epsilon_i$ varies because it goes over every single power of $p$ which is present in $x_i$. Since $x_i$ is a subseries of $\sum_{n\in\N}\frac{f(n)}{n^s}$, it will also converge absolutely for $\sigma>\sigma_a$ and any rearrangement of its terms does not change the value of the sum. Since $f$ is multiplicative and each $p_i$ is coprime to the others by definition, we have 

$$e_k=\sum_{1\leqslant i_1< i_2 \cdots < i_k \leqslant n }\frac{f\left(p_{i_1}^{\epsilon_{i_1}}p_{i_2}^{\epsilon_{i_2}}\cdots p_{i_k}^{\epsilon_{i_k}}\right)}{\left(p_{i_1}^{\epsilon_{i_1}}p_{i_2}^{\epsilon_{i_2}}\cdots p_{i_k}^{\epsilon_{i_k}}\right)^s}.$$

If we take an arbitrary natural number $n$ with $k$ distinct prime factors, it will be present in the sum with the $k$\textsuperscript{th} symmetric function, $e_k$. The $k$\textsuperscript{th} symmetric function contains every natural number with $k$ prime factors, since $k$ dictates the number of terms that are multiplied together to form every term in $e_k$. The multiplicity also doesn't matter, since that varies with $\epsilon_i$ which is independent of $k$. 

We also know that by the fundamental theorem of arithmetic, there is a bijection between the natural numbers and the products of distinct primes with any multiplicity. This means that every product of distinct primes in the expression for $e_k$ corresponds to a natural number $n$. Taking it all together, it follows that $e_k$ goes over every natural number $n$ with $k$ distinct prime factors. Rewriting each product of primes as $n$ then gives equation \eqref{inf2}. Since $\sum_{n \in S_k}\frac{f(n)}{n^s}$ is a subseries of $\sum_{n\in\N}\frac{f(n)}{n^s}$, it will also converge absolutely for $\sigma>\sigma_a$.
\end{proof}

We then have $\sum_{k=1}^{\infty}k e_k = \sum_{k=1}^{\infty}k\sum_{n \in S_k}\frac{f(n)}{n^s}$, where $S_k$ is defined by \eqref{defS}. This means that as $k$ goes from $1$ to $\infty$ the sum of each $k e_k$ from the left-hand side can be interpreted to go over every natural number except $1$ because they have been partitioned based on how many distinct prime factors they have. The series fails to sum over $n=1$, which does not have any prime factors, but $\omega(1)=0$ so this does not affect the sum in any way.

We can also see $\omega(n)$ is the weight that's represented by $k$ since we can bring it inside the inner sum as $\omega(n)$ and rewrite the double sum as a sum over the natural numbers. We can also change the bottom limit from $k=1$ to $k=0$ since the $k=0$ term is $0$. This gives
$$\sum_{k=1}^{\infty}k\sum_{S_k}\frac{f(n)}{n^s}= \sum_{k=0}^{\infty}\sum_{S_k}\omega(n)\frac{f(n)}{n^s}= \sum_{n\in\N}\frac{\omega(n)f(n)}{n^s}.$$
The inner sum converges absolutely for $\sigma>\sigma_a$, but the sum over $n$ does not converge on this half plane in general. This rearrangement is valid if $\sum_{n\in\N}\frac{\omega(n)f(n)}{n^s}$ converges absolutely. However, as a Dirichlet series it is guaranteed to have an abscissa of absolute convergence and therefore the rearrangement is valid for some $\sigma$.

Simplifying \eqref{mainpos} finally shows that
$$\sum_{n\in\N}\frac{\omega(n)f(n)}{n^s} = \left(\sum_{n\in\N}\frac{f(n)}{n^s}\right)\left(\sum_{p}\frac{a_p}{1+a_p}\right).$$ 
The left-hand side has the same convergence criteria as the right-hand side. Therefore if $\sum_{n\in\N}\frac{f(n)}{n^s}$ has an abscissa of absolutely convergence $\sigma_a$ and $\sum_{p}\frac{a_p}{1+a_p}$ converges absolutely for some $\sigma>\sigma_b$, the left-hand side will converge absolutely for $\sigma>\max\left(\sigma_a, \sigma_b\right)$. This shows that weighting the terms of the Dirichlet series of any multiplicative function $f(n)$ by $\omega(n)$ multiplies the original series by a sum of $f$ over primes.
\end{proof}

\begin{thm} Let $f(n)$ be a completely multiplicative function, $s\in\C$, and $n\in\N$. If $\sum_{n\in\N}\frac{f(n)}{n^s}$ has an abscissa of absolute convergence $\sigma_a$, then
\begin{equation}\label{inf3}
\sum_{n\in\N}\frac{\omega(n)f(n)}{n^s} =  \left(\sum_{n\in\N}\frac{f(n)}{n^s}\right)\left(\sum_{p}\frac{f(p)}{p^s}\right),
\end{equation}
which converges absolutely for $\sigma>\sigma_a$.
\end{thm}
\begin{proof}
If $\sigma>\sigma_a$, then $\sum_{m=1}^{\infty} \frac{f(p_i^m)}{{p_i}^{ms}}$ converges absolutely and therefore  $\left|\frac{f(p)}{p^s}\right|<1$. Then
$$x_i = a_{p_i}= \sum_{m=1}^{\infty} \frac{f(p_i^m)}{{p_i}^{ms}} =  \sum_{m=1}^{\infty} \frac{f(p_i)^m}{{p_i}^{ms}} = \frac{1}{1-{\frac{f(p_i)}{{p_i}^s}}}-1.$$
Replacing $a_p$ by $\left({1-{\frac{f(p)}{{p}^s}}}\right)^{-1}-1$ in \eqref{inf1} and simplifying completes the proof. We also note that $\sum_{n\in\N}\frac{f(p)}{p^s}$ is a subseries of $\sum_{n\in\N}\frac{f(n)}{n^s}$, so it will also converge absolutely for $\sigma>\sigma_a$ and we can simplify our convergence criterion. 
\end{proof}

\begin{thm} Let $f(n)$ be a multiplicative function. Let $\sum_{n\in\N}\frac{|\mu(n)|f(n)}{n^s}$ and $\sum_{p}\frac{f(p)}{p^s+f(p)}$ both converge absolutely for $\sigma>\sigma_a$. Assume that for all prime $p$ and $s \in \C$ such that $\sigma > \sigma_a$ we have $f(p) \neq - p^s$. Then
\begin{equation}\label{inf5}
\sum_{n\in\N}\frac{|\mu(n)|\omega(n)f(n)}{n^s} = \left(\sum_{n\in\N}\frac{|\mu(n)|f(n)}{n^s}\right) \left( \sum_{p}\frac{f(p)}{p^s+f(p)}\right),
\end{equation}
which converges absolutely for $\sigma>\sigma_a$.
\end{thm}
\begin{proof}
To begin, we note that a product of multiplicative functions is also multiplicative. Letting $f(n) = |\mu(n)|g(n)$ in \eqref{inf1}, where $g$ is any multiplicative function which guarantees that $f$ is multiplicative, we can simplify $a_p$. We have that $|\mu(p^m)|$ is $0$ for $m\geq 2$ and $1$ for $m=1$, since $|\mu(n)|$ is the characteristic function of the squarefree integers. If $m=1$, we also have that $f(p)=|\mu(p)|g(p) = g(p)$, which means that $a_p=\sum_{m=1}^{\infty}\frac{f({p}^m)}{{p}^{ms}} = \frac{g(p)}{p^s}$. We still retain the $a_p=\frac{g(p)}{p^s}\neq -1$ condition. Assuming that the sum over primes converges, substituting into \eqref{inf1} gives
$$\sum_{n\in\N}\frac{|\mu(n)|\omega(n)g(n)}{n^s} = \left(\sum_{n\in\N}\frac{|\mu(n)|g(n)}{n^s}\right) \left(\sum_{p}\frac{\frac{g(p)}{p^s}}{1+\frac{g(p)}{p^s}}\right).$$
We utilize the same convergence criterion as Theorem \ref{infmain}. Simplifying the fraction and letting $g$ be represented by $f$ in order to maintain consistent notation completes the proof.
\end{proof}

\begin{thm} Let $f(n)$ be a multiplicative function. Let $\sum_{n\in\N}\frac{\mu(n)f(n)}{n^s}$ and $\sum_{p}\frac{f(p)}{p^s-f(p)}$ both converge absolutely for $\sigma>\sigma_a$. Assume that for all prime $p$ and $s \in \C$ such that $\sigma > \sigma_a$ we have $f(p) \neq p^s$. Then
\begin{equation}\label{inf6}
\sum_{n\in\N}\frac{\mu(n)\omega(n)f(n)}{n^s} = \left(\sum_{n\in\N}\frac{\mu(n)f(n)}{n^s}\right)\left(\sum_{p}\frac{f(p)}{f(p)-p^s}\right),
\end{equation}
which converges absolutely for $\sigma>\sigma_a$.
\end{thm}
\begin{proof}
We let $f(n) = \mu(n)g(n)$ in \eqref{inf1}, where $g$ is any multiplicative function which guarantees that $f$ is multiplicative. We can then simplify $a_p$, since $a_p:=\sum_{m=1}^{\infty}\frac{f({p}^m)}{{p}^{ms}} = -\frac{g(p)}{p^s}$. We still retain the $a_p = -\frac{g(p)}{p^s}\neq -1$ condition. Assuming that the sum over primes converges, substituting into \eqref{inf1} gives
$$\sum_{n\in\N}\frac{\mu(n)\omega(n)g(n)}{n^s} = \left(\sum_{n\in\N}\frac{\mu(n)g(n)}{n^s}\right)\left(\sum_{p}\frac{-\frac{g(p)}{p^s}}{1-\frac{g(p)}{p^s}}\right).$$
We utilize the same convergence criterion as Theorem \ref{infmain}.  Simplifying the fraction and letting $g$ be represented by $f$ in order to maintain consistent notation completes the proof.
\end{proof}

We note that the following proposition, the simplest application of \eqref{inf1}, can be found in \cite[(D-17)]{Catalog}.
\begin{thm} Let $\zeta(s)$ be the Riemann zeta function. For $s\in\C$ such that $\sigma>1$ ,
\begin{equation}
\sum_{n\in\N}\frac{\omega(n)}{n^s} = \zeta(s) P(s).
\end{equation}
\end{thm}
\begin{proof}
We let $f(n) = 1$ in \eqref{inf3}, since this is a completely multiplicative function. We then note that $\sum_{p}\frac{1}{p^s}$ is the prime zeta function. The zeta and prime zeta functions both converge absolutely for $\sigma>1$, so the left-hand side will too.
\end{proof}

\begin{thm} Let $\lambda(s)$ be Liouville's function. For $s\in\C$ such that $\sigma>1$,
\begin{equation}
\sum_{n\in\N}\frac{\omega(n)\lambda(n)}{n^s} = -\frac{\zeta(2s)}{\zeta(s)} P(s).
\end{equation}
\end{thm}
\begin{proof}
We let $f(n) = \lambda(n)$ in \eqref{inf3}. Liouville's function, found in \cite[(27.2.13)]{NIST:DLMF}, is completely multiplicative. We note that $\lambda(p)=-1$ for every prime $p$, since they trivially only have a single prime divisor with multiplicity $1$. This gives
$$\sum_{n\in\N}\frac{\omega(n)\lambda(n)}{n^s} =\left( \sum_{n\in\N}\frac{\lambda(n)}{n^s}\right)\left( \sum_{p}\frac{(-1)}{p^s}\right).$$
We then note that \cite[(27.4.7)]{NIST:DLMF} gives $\sum_{n\in\N}\frac{\lambda(n)}{n^s}=\frac{\zeta(2s)}{\zeta(s)}$ and states that it converges for $\sigma>1$, which completes the proof since both the zeta and prime zeta functions converge for $\sigma>1$.
\end{proof}

\begin{thm} Let $\chi(n)$ denote a Dirichlet character. For $s\in\C$ such that $\sigma>1$ ,   
\begin{equation}
\sum_{n\in\N}\frac{\omega(n)\chi(n)}{n^s} = L(s,\chi) \sum_{p}\frac{\chi(p)}{p^s}.
\end{equation}
\end{thm}
\begin{proof}
We let $f(n) = \chi(n)$ in \eqref{inf3}. Here $\chi(n)$ is a Dirichlet character, found in \cite[(27.8.1)]{NIST:DLMF}, which is a completely multiplicative function that is periodic with period $k$ and vanishes for $(n,k)>1$. Substituting it in \eqref{inf3} gives
$$\sum_{n\in\N}\frac{\omega(n)\chi(n)}{n^s} = \left(\sum_{n\in\N}\frac{\chi(n)}{n^s}\right)\left( \sum_{p}\frac{\chi(p)}{p^s}\right).$$
We then note that a Dirichlet $L$-series, an important number theoretic series, is defined in \cite[(25.15.1)]{NIST:DLMF} as $L(s, \chi) = \sum_{n\in\N}\frac{\chi(n)}{n^s}.$
Simplifying to write the sum over $n$ as an $L$-series while noting that \cite[(25.15.1)]{NIST:DLMF} states that an $L$-series converges absolutely for $\sigma>1$ completes the proof, since both the $L$ series and prime $L$ series converge absolutely for $\sigma>1$.
\end{proof}

\begin{thm} For $s\in\C$ such that $\sigma>1$ ,   
\begin{equation}
\sum_{n\in\N}\frac{|\mu(n)|\omega(n)}{n^s} = \frac{\zeta(s)}{\zeta(2s)} P(s,1).
\end{equation}
\end{thm}
\begin{proof}
We let $f(n) = 1$ in \eqref{inf5} and note that \cite[(27.4.8)]{NIST:DLMF} states that $\sum_{n\in\N}\frac{|\mu(n)|}{n^s} = \frac{\zeta(s)}{\zeta(2s)}$ and that this converges for $\sigma>1$. We note that \eqref{primezetacor} states that $P(s,1)$ will also converge absolutely for $\sigma>1$, which completes the proof.
\end{proof}

\begin{thm}  For $s\in\C$ such that $\sigma>1$ ,   
\begin{equation}
\sum_{n\in\N}\frac{\mu(n)\omega(n)}{n^s} = -\frac{1}{\zeta(s)} P(s,-1).
\end{equation}
\end{thm}
\begin{proof}
We let $f(n) = 1$ in \eqref{inf6} and note that \cite[(27.4.5)]{NIST:DLMF} states $\sum_{n\in\N}\frac{\mu(n)}{n^s}=\frac{1}{\zeta(s)}$ and that it converges for $\sigma>1$. We note that \eqref{primezetacor} states that $P(s,-1)$ will also converge absolutely for $\sigma>1$, which completes the proof.
\end{proof}

While the previous sums have involved completely multiplicative functions or convolutions with the M{\"o}bius function, we can sometimes directly evaluate $\sum_{m=1}^{\infty}\frac{f({p}^m)}{{p}^{ms}}$. Taking \eqref{inf1} but converting $\sum_{n\in\N}\frac{f(n)}{n^s}$ back to its Euler product means that 
$$\sum_{n\in\N}\frac{\omega(n)f(n)}{n^s} = \prod_{p}\left(1+a_p\right) \sum_{p}\left(\frac{a_p}{1+a_p}\right),$$
where $a_p:=\sum_{m=1}^{\infty}\frac{f({p}^m)}{{p}^{ms}}$. We can extract the coefficient $a_p$ through a variety of methods.

\begin{thm} For $s\in\C$ such that $\sigma>1$,   
\begin{equation}
\sum_{n\in\N}\frac{\omega(n)2^{\omega(n)}}{n^s} = 2\frac{\zeta^2(s)}{\zeta(2s)} P(s,1).
\end{equation}
\end{thm}
\begin{proof}
We know from \cite[(27.4.9)]{NIST:DLMF} that $\sum_{n\in\N}\frac{2^{\omega(n)}}{n^s}=\frac{\zeta^2(s)}{\zeta(2s)}$ and that it converges for $\Re(s)>1$, where $2^{\omega(n)}$ is the number of squarefree divisors of $n$.
We now directly evaluate $a_p$, summing it as a geometric series. We note that $2^{\omega\left(p^m\right)}=2$, since $p^m$ trivially has a single distinct prime factor. Substituting into the formula for $a_p$ shows that 
$$a_p=\sum_{m=1}^{\infty}\frac{2^{\omega\left(p^m\right)}}{{p}^{ms}}=\sum_{m=1}^{\infty}2{\left(\frac{1}{p^s}\right)}^m=\frac{2}{p^s-1}.$$
Then $\frac{a_p}{1+a_p}=\frac{2}{p^s+1}$.
Substituting into \eqref{inf1} shows that 
$$\sum_{n\in\N}\frac{\omega(n)2^{\omega(n)}}{n^s}= \left(\sum_{n\in\N}\frac{2^{\omega(n)}}{n^s}\right)\left( \sum_{p}\frac{2}{p^s+1}\right).$$
Rewriting the right-hand side in terms of zeta and prime zeta functions while noting that they will both converge if $\sigma>1$ completes the proof.
\end{proof}

\begin{thm}\label{infJordan} Let $J_k(n)$ denote Jordan's totient function. For $s,k\in\C$ such that $\sigma>\max\left(1,1+\Re(k)\right)$,
\begin{equation}
\sum_{n\in\N}\frac{\omega(n)J_k(n)}{n^s} =\frac{\zeta(s-k)}{\zeta(s)} \sum_{p}\frac{p^k-1}{p^s-1}.
\end{equation}
\end{thm}
\begin{proof}
Taking $a_p=\frac{p^k-1}{p^s-p^k}$, we have $(1+a_p)={\left(1-\frac{1}{p^s}\right)}{\left(1-\frac{1}{p^{s-k}}\right)}^{-1}$ and $\frac{a_p}{1+a_p} = \frac{p^k-1}{p^s-1}$. Substituting into \eqref{inf1} shows that
$$\sum_{n\in\N}\frac{\omega(n)J_k(n)}{n^s} =\prod_{p}{\left(1-\frac{1}{p^s}\right)}{\left(1-\frac{1}{p^{s-k}}\right)}^{-1}\left( \sum_{p}\frac{p^k-1}{p^s-1}\right).$$
The series $\sum_{p}\frac{p^k-1}{p^s-1}$ does not have a representation as a sum of prime zeta and shifted prime zeta functions in general, but in special cases such as $s = 2k$, $k>1$, it does. Rewriting in terms of zeta functions and taking convergence criteria based on \eqref{primezetacor2} completes the proof.
\end{proof}

We state several theorems without proof. They are all special cases of Theorem \ref{infmain} and can be proved similarly to Theorem \ref{infJordan}. In each case we begin with a known Euler product for a function $f(n)$, then use that to extract $a_p$.

\begin{thm} Let $\sigma_k(n)$ denote the sum of the $k$th powers of the divisors of $n$. For $s,k\in\C$ such that $\sigma>\max\left(1,1+\Re(k)\right)$,
\begin{equation}\label{inf12}
\sum_{n\in\N}\frac{\omega(n)\sigma_k(n)}{n^s} = \zeta(s)\zeta(s-k) \left(P(s)+P(s-k)-P(2s-k)\right).
\end{equation}
\end{thm}

\begin{thm} Let $d(n)$ equal the number of divisors of n. For $s\in\C$ such that $\sigma>1$,   
\begin{equation}
\sum_{n\in\N}\frac{\omega(n)d(n^2)}{n^s} = \frac{\zeta^3(s)}{\zeta(2s)} \left(4P(s,1)-P(s)\right).
\end{equation}
\end{thm}

\begin{thm} Let $d(n)$ equal the number of divisors of n.  For $s\in\C$ such that $\sigma>1$,    
\begin{equation}
\sum_{n\in\N}\frac{\omega(n)d^2(n)}{n^s} = \frac{\zeta^4(s)}{\zeta(2s)} \left(8P(s,1)+P(2s)-4P(s)\right).
\end{equation}
\end{thm}

\section{Extensions}\label{extensions}
Lastly, we show how to generalize the methods of this paper to second and higher order derivatives. Starting with the genrating product for $e_k$, that $\prod_{k=1}^n (1+tx_i) = \sum_{k = 0}^{n}e_k t^k$, we study the action of the differential operator $D:=x\frac{d}{dx}$. Applying it once yields
$$D\left(\prod_{k=1}^n (1+tx_i)\right)=\sum_{k = 1}^{n}k e_k t^k,$$
from which we recover the familiar \eqref{main} after evaluating at $t=1$. Applying $D$ a second time results in 
\begin{align}
\sum_{k = 1}^{n}k^2 e_k t^k &=D\left(x\prod_{k=1}^n (1+tx_i)   \sum_{k=1}^n \frac{x_i}{1+t x_i}\right) \\&=\left(\prod_{i = 1}^{n}(1+x_i)\right)\left( \left(\sum_{i = 1}^{n}\frac{x_i}{1+x_i}\right)^2+\sum_{i=1}^{n}\frac{x_i}{\left(1+x_i\right)^2}\right).
\end{align}
Letting $x_i =a_p:= \frac{f(p_i)}{{p_i}^s}+\frac{f({p_i}^2)}{{p_i}^{2s}}+\cdots = \sum_{m=1}^{\infty}\frac{f({p_i}^m)}{{p_i}^{ms}}$, and taking $t \rightarrow 1$ and $n\rightarrow \infty$ gives
\begin{equation}
\sum_{n\in\N}\frac{{\omega(n)}^2f(n)}{n^s}=\left(\sum_{n\in\N} \frac{f(n)}{n^s}\right)\left( \left(\sum_{p}\frac{a_p}{1+a_p}\right)^2+\sum_{p}\frac{a_p}{\left(1+a_p\right)^2}\right).
\end{equation}
For a completely multiplicative function, summing $a_p$ as a geometric series reduces this to
\begin{equation}
\sum_{n\in\N}\frac{{\omega(n)}^2f(n)}{n^s}=\left(\sum_{n\in\N} \frac{f(n)}{n^s}\right)\left( \left(\sum_{p}\frac{f(p)}{p^s}\right)^2+\sum_{p}\frac{f(p)}{p^s}-\sum_{p}\left(\frac{f(p)}{p^s}\right)^2\right).
\end{equation}
As an example, letting $f(n)=1$ gives
\begin{equation}
\sum_{n\in\N}\frac{{\omega(n)}^2}{n^s}=\zeta(s)\left( P^2(s)+P(s)-P(2s)\right),
\end{equation}
which converges for $\Re(s)>1$. In general, applying $D$ to $E(t)$ $k$ times will result in a Dirichlet series of the form $\sum_{n\in\N}\frac{\omega(n)^kf(n)}{n^s}$.



These theorems apply to any multiplicative functions. Together, this allows for a large class of infinite and divisor sums weighted by $\omega(n)$ to be addressed for the first time. This leads to some surprising results such as that the Dirichlet series for products of $\omega(n)$ and other multiplicative functions often have a convenient closed form expression in terms of zeta and prime zeta functions. The methods of this paper also suggest an obvious generalization; taking the $k$th derivative of $E(t)$ will lead to finite and infinite sums involving $\omega(n)^k$. This also suggests deep connections between the theory of symmetric functions and Dirichlet series, since with the right choice of $x_i$ we can interpret a Dirichlet series as a sum over symmetric polynomials. Different identities for symmetric polynomials with correspond to general expressions for Dirichlet series weighted by different functions. Due to this, a systematic study of identities for symmetric functions should correspond to identities for Dirichlet series.

\subsection*{Acknowledgements}
Many thanks to my research supervisors at the National Institute of Standards and Technology, Gaithersburg, Maryland; Guru Khalsa, Mark Stiles, Kyoung-Whan Kim, and Vivek Amin have at various times helped me through random issues. Many thanks as well to Howard Cohl for his invaluable style tips and general mathematical expertise. Without him this paper wouldn't have gotten out of the ground. Bruce Berndt of the University of Illinois at Urbana-Champaign and Krishna Alladi of the University of Central Florida also provided valuable feedback. Many thanks also go out to my teachers - William Wuu, Joseph Boettcher, Colleen Adams, Joshua Schuman, Jamie Andrews, and others - at Quince Orchard High School, who have miraculously put up with me through the years.


\begin{thebibliography}{1}

\bibitem{HardyRamanujan}
Hardy, G. H. and Ramanujan, S.
\newblock {The normal number of prime factors of a number {$n$} [{Q}uart. {J}. {M}ath. {\bf 48} (1917), 76--92]}.
\newblock {\em {Collected papers of {S}rinivasa {R}amanujan}}.
\newblock {AMS Chelsea Publ., Providence, RI, 2000}.

\bibitem{ErdosKac}
{Erd{\"o}s, P. and Kac, M.}.
\newblock The {G}aussian law of errors in the theory of additive number theoretic functions.
\newblock {\em American Journal of Mathematics}, 62:738--742, 1940.

\bibitem{NIST:DLMF}
{NIST Digital Library of Mathematical Functions}.
\newblock {\url{http://dlmf.nist.gov}, Release 1.0.9 of 2014-08-29}.
\newblock {Online companion to \cite{NIST}}.

\bibitem{NIST}
{Olver, F.~W.~J. and Lozier, D.~W. and Boisvert, R.~F. and Clark, C.~W.}, editors.
\newblock {\em {NIST Handbook of Mathematical Functions}}.
\newblock {Cambridge University Press}, {New York, NY}, {2010}.
\newblock {Print companion to \cite{NIST:DLMF}}.

\bibitem{Apostol}
Apostol, T.~M.
\newblock {\em {Introduction to Analytic Number Theory}}.
\newblock {Springer-Verlag}, {New York, NY}, {1976}.

\bibitem{PrimeZeta}
Fr{\"o}berg, C.-E.
\newblock On the prime zeta function.
\newblock {\em Nordisk Tidskr. Informationsbehandling (BIT)}, 8:187--202, 1968.

\bibitem{SurveyMultFuncs}
Mathar, R.~J. 
\newblock Survey of Dirichlet series of multiplicative arithmetic functions.
\newblock {\em {Pre-print}}, 2012.
\newblock {http://arxiv.org/abs/1106.4038}.

\bibitem{Catalog}
{Gould, H. W. and Shonhiwa, T.}.
\newblock {A catalog of interesting Dirichlet series}.
\newblock {\em Missouri J. Math. Sci.}, 1:2--18, 2008.

\bibitem{Macdonald}
{I. G. Macdonald},
\newblock {\em {Symmetric functions and {H}all polynomials},}
\newblock{2nd ed.,}
\newblock {Oxford University Press, New York}, {1995}.
\end{thebibliography}
\end{document}